\documentclass[reqno]{amsart}
\usepackage{amsmath}

\usepackage{graphicx}
\usepackage{amsfonts}
\usepackage{mathptmx}
\usepackage{amssymb}
\usepackage{hyperref}
\usepackage{xcolor}
\usepackage{epstopdf}
\usepackage[a4paper,top=3.5cm,bottom=3cm,left=2.5cm,right=2.5cm]{geometry}
\makeindex
\begin{document}
\title[Solvability of a sixth-order boundary value problem]{ Solvability of a sixth-order boundary value problem with multi-point and multi-term integral boundary conditions}
\author[N. Houari]{Nourredine Houari}
\address{
Laboratory of Fundamental and Applied Mathematics of Oran (LMFAO), Department of Mathematics, University of Oran 1, Oran, Algeria.}
\author[F. Haddouchi]{Faouzi Haddouchi}
\address{
Department of Physics, University of Sciences and Technology of
Oran-MB, Oran, Algeria.
\newline
And
\newline
Laboratory of Fundamental and Applied Mathematics of Oran (LMFAO), University of Oran 1, Oran,  Algeria.}
\email{noureddinehouari.mi@gmail.com}
\email{fhaddouchi@gmail.com}

\maketitle \numberwithin{equation}{section}
\newtheorem{theorem}{Theorem}[section]
\newtheorem{lemma}[theorem]{Lemma}
\newtheorem{definition}[theorem]{Definition}
\newtheorem{proposition}[theorem]{Proposition}
\newtheorem{corollary}[theorem]{Corollary}
\newtheorem{remark}[theorem]{Remark}
\newtheorem{exmp}{Example}[section]
\begin{abstract}
 This paper aims to investigate the existence and uniqueness of solutions for a sixth order differential equation involving nonlocal and integral boundary conditions. Firstly, we obtain the properties of the relevant Green's functions. The existence result of at least one nontrivial solution is obtained by applying the Krasnoselskii-Zabreiko fixed point theorem. Moreover, we also establish the existence of unique solution to the considered problem via H\"{o}lder and Minkowski inequalities and Rus's theorem. Finally, two examples are included to show the applicability of our main results.
\end{abstract}
\section{Introduction}
Differential equations have been studied extensively in various fields of the science and engineering, see the books \cite{Book1,Book5,Book3,Book4,Book2}. In recent years, the existence and uniqueness of solutions of boundary value problems with various multipoint and integral boundary conditions has been studied broadly by many researchers using a variety of approaches, tools and techniques and obtained many meaningful results.  See \cite{Akram,Ballem,Ben,Bon1,Bon2,Che,Guen,Hadd4,Hadd5,Hadd1,Hadd3,Hadd2,Li2,Li1,Liu,Mart,Odda1,Odda2,Panos,Sidiq1,Sidiq2} and references cited therein.

In particular, for sixth-order problems, Li \cite{Li2} studied, by using spectral theory of operators and fixed point theorem in cone, the existence and multiplicity of positive solutions for the boundary value problem consisting of the equation
\begin{equation*}
	u^{(6)} + A(t)u^{(4)} + B(t)u^{(2)} + C(t)u + f(t,u) = 0,
\end{equation*}
with the boundary conditions\\
\begin{equation*}
	u(0) = u(1) = u''(0)= u''(1) = u^{(4)}(0) = u^{(4)}(1) = 0,
\end{equation*}
where $A(t), B(t), C(t) \in C([0,1])$ and $f(t,u)\in C([0,1]\times [0,\infty),[0,\infty))$.

In \cite{bek1}, Bekri and Benaicha interested, via Leray-Schauder nonlinear alternative, in the existence of a nontrivial solution for the nonlinear boundary value problem
\begin{equation*}
	-u^{(6)}(t) + f(t,u(t),u''(t)) = 0,\ 0<t<1,
\end{equation*}
\begin{equation*}
	u(0)=u'(0)=u''(0)=0,\ u'''(1)=u^{(4)}(1)=u^{(5)}(1)=0,
\end{equation*}
where $f\in C([0,1]\times\mathbb{R}^2,\mathbb{R}).$

In 2019, Yang \cite{yang} obtained, by means of Krasnoselskii fixed point theorem, sufficient conditions for the existence and nonexistence of positive solutions for the following sixth order boundary value problem
\begin{equation*}
	u^{(6)}(t) + g(t)f(u(t)) = 0,\ 0\leq t\leq 1,
\end{equation*}
\begin{equation*}
	u(0)=u'(0)=u''(0)=0,\ u'(1)=u'''(1)=u^{(5)}(1)=0,
\end{equation*}
where $f\in C([0,\infty),[0,\infty))$ and $g:[0,1]\rightarrow [0,\infty)$ is continuous with $g(t)\neq 0$ on $[0,1]$.

A year later, the authors \cite{Mart} used Avery Peterson's theorem to investigate the existence of multiple solutions and the existence and uniqueness via a new numerical method based on Banach contraction principle for the following nonlinear sixth order boundary value problem
\begin{equation*}
	u^{(6)}+f(t,u)=0,\ 0<t<1,
\end{equation*}
\begin{equation*}
	u(0) = u'(0) = u''(0) = 0, u'(1) = u'''(1) = u^{(5)}(1) = 0,
\end{equation*}
where $f\in C(\mathbb{R}^2, \mathbb{R}).$

Very recently, Bonnano et al.\cite{Bon1} used critical point theory to discuss the existence of at least one nontrivial solution for the nonlinear sixth order boundary value problem
\begin{equation*}
	-u^{(6)} + Au^{(4)} - Bu''+ Cu = \lambda f(x,u),\ x\in[a,b],
\end{equation*}
\begin{equation*}
	u(a) = u(b) = u''(a) = u''(b) = u^{(4)}(a) = u^{(4)}(b) = 0,
\end{equation*}
where, $\lambda>0, A, B$ and $C$ are constants and $f :[a,b]\times\mathbb{R}\rightarrow\mathbb{R}$ is a function.

In this paper, motivated by the above works, we investigate the following sixth-order boundary value problem with integral multi-point boundary condition
\begin{equation}\label{eq001}
u^{(6)}(t) + f(t,u(t)) = 0,\ t\in (0,1),
\end{equation}
\begin{equation}\label{eq002}
 u^{(j)}(0)= u''(1)=0,\ 1\leq j\leq 4 ,\
 u(0)= \sum_{i=1}^{m-1}\alpha_i\int^{\eta_{i+1}}_{\eta_i} u(s)ds + \sum_{i=1}^{m}\beta_iu(\eta_i),
\end{equation}
 where
 \begin{itemize}
\item[(H1)] $ f \in C(J\times\mathbb{R}, \mathbb{R})$ with $J= [0,1]$;
\item[(H2)] $ \alpha_i,\beta_i\in \mathbb{R}$,\ and $ 0<\eta_1<...<\eta_m\leq 1,\ m\geq 2$.
 \end{itemize}
First, we establish some properties of Green's functions, and then we show the existence of at least one nontrivial solution of the nonlinear problem \eqref{eq001}-\eqref{eq002} via Krasnoselskii-Zabreiko's fixed point theorem. To provide the uniqueness result, we use two metrics under Rus's fixed point theorem.

\section{Preliminaries}
\begin{definition}
An operator $A: X \rightarrow X$ is completely continuous if it is continuous and maps bounded sets into relatively compact sets.
\end{definition}
The H\"{o}lder and Minkowski inequalities are key results in our discussion in Section 3.\\
\begin{theorem}(H\"{o}lder inequality)
Let $p>1$ and $q>1$ be conjugate exponents (i.e., $\frac{1}{p} + \frac{1}{q} = 1$), $f\in L^{p}(J)$ and $g\in L^{q}(J)$. Then $fg\in L^{1}(J)$ and $\|fg\|_{1}\leq\|f\|_{p}\|g\|_{q}$.
Further, if $f \in L^{1}(J)$ and $g\in L^{\infty}(J)$. Then $fg\in L^{1}(J)$ and $\|fg\|_{1}\leq\|f\|_{1}\|g\|_{\infty}$.
\end{theorem}
\begin{theorem}(Minkowski's inequality)
Assume that $1\leq p<\infty$. If $f,g\in L^{p}(J)$, then $f+g\in L^{p}(J)$ and we have the triangle inequality $\|f+g\|_{p}\leq\|f\|_{p}+\|g\|_{p}$.
\end{theorem}

To prove our main results, we need to present the main tools to be used.
\begin{theorem}(Rus's Theorem) \label{Thm Rus} \cite{Rus}
Let $X$ be a nonempty set and let $d$ and $\delta$ be two metrics on $X$ such that ($X,d$) forms a complete metric space. If the mapping $A : X \rightarrow X$ is continuous with respect to $d$ on $X$ and \
\begin{eqnarray*}
d(Au_1,Au_2) \leq \xi \delta(u_1,u_2),
\end{eqnarray*}
for some $\xi>0$.\
And for all $u_1,u_2\in X$
\begin{eqnarray*}
\delta(Au_1,Au_2) \leq \nu \delta(u_1,u_2),
\end{eqnarray*}
for some $0<\nu<1$, then, there exists a unique $u^{*}\in X$ such that $Au^{*}=u^{*}.$
\end{theorem}
\begin{theorem}(Krasnoselskii-Zabreiko fixed point Theorem) \label{Thm Kras-Zabr} \cite{Kras}
Let $(X,\lVert . \rVert)$ be a Banach space, and $A : X \rightarrow X$ be a completely continuous operator. Assume that  $L : X \rightarrow X$ is a bounded linear operator such that 1 is not an eigenvalue of $L$ and
\begin{equation} \label{eqZK}
\lim_{\lVert u\rVert \rightarrow \infty}\frac{\lVert Au-Lu\rVert}{\lVert u\rVert} =0.
\end{equation}
Then $A$ has a fixed point in $X$.	
\end{theorem}
First of all, we consider the linear equation
\begin{equation}\label{eq003}
u^{(6)}(t) + h(t)= 0, \ t \in (0,1),
\end{equation}
subject to the boundary conditions \eqref{eq002}.\\

For convenience, we denote $ \mu = 1-\big(\sum_{i=1}^{m-1}\alpha_i(\eta_{i+1}-\eta_i) + \sum_{i=1}^{m}\beta_i\big)$.\\
Next, we will present the following auxiliary results.

\begin{lemma} \label{lem1}
Assume that $\mu \neq 0$ and $h \in C(J,\mathbb{R})$, then the unique solution $u$ of the boundary value problem \eqref{eq003}-\eqref{eq002} is given by
 \begin{equation} \label{eqsol}
 u(t)= \int_0^1\bigg[G(t,s) + \frac{1}{\mu}\sum_{i=1}^{m-1}\alpha_i\Big( K(\eta_{i+1},s) - K(\eta_i,s)\Big)+\frac{1}{\mu}\sum_{i=1}^{m}\beta_iG(\eta_i,s) \bigg]h(s)ds,\ \forall t\in J,
 \end{equation}
 where
\begin{equation} \label{eq004}
 G(t,s) = \frac{1}{5!}
 \begin{cases}
 t^5(1-s)^3-(t-s)^5,& 0\leq s\leq t\leq 1 ,\\
 t^5(1-s)^3,& 0\leq t\leq s\leq 1,\\
 \end{cases}
\end{equation}
\begin{equation} \label{eq004(2)}
K(t,s) = \frac{1}{6!}
\begin{cases}
t^6(1-s)^3-(t-s)^6,& 0\leq s\leq t\leq 1 ,\\
t^6(1-s)^3,& 0\leq t\leq s\leq 1,\\
\end{cases}
\end{equation}
\end{lemma}
\begin{proof}
In fact, if $u(t)$ is a solution of the boundary value problem \eqref{eq003}-\eqref{eq002}, then we obtain
\begin{equation} \label{eq005}
u(t) = -\frac{1}{5!}\int_{0}^{t}(t-s)^5h(s)ds+\frac{\kappa_1}{5!}t^5+\frac{\kappa_2}{4!}t^4+\frac{\kappa_3}{6}t^3+\frac{\kappa_4}{2}t^2+\kappa_5t+\kappa_6,
\end{equation}
where $\kappa_i \in\mathbb{R},\ i\in\{1,\ldots,5\}$. By the boundary conditions \eqref{eq002}, we get
\begin{eqnarray*}
\kappa_2=\kappa_3=\kappa_4=\kappa_5=0,\  \kappa_1=\int_{0}^{1}(1-s)^3h(s)ds.
\end{eqnarray*}
Further
\begin{eqnarray*}
\kappa_6&=& u(0)\\
&=&\sum_{i=1}^{m-1}\alpha_i\int^{\eta_{i+1}}_{\eta_i} u(\tau)d\tau + \sum_{i=1}^{m}\beta_iu(\eta_i)\\
&=& \sum_{i=1}^{m-1}\alpha_i\int^{\eta_{i+1}}_{\eta_i} \bigg[-\frac{1}{5!}\int_{0}^{\tau}(\tau -s)^5 h(s)ds +\frac{\tau^5}{5!}\int_{0}^{1}(1-s)^3h(s)ds + \kappa_6\bigg]d\tau\\
&&+
\sum_{i=1}^{m}\beta_i\bigg[-\frac{1}{5!}\int_{0}^{\eta_i}(\eta_i-s)^5 h(s)ds +\frac{\eta_i^5}{5!}\int_{0}^{1}(1-s)^3h(s)ds + \kappa_6\bigg],
\end{eqnarray*}
and thus, we get
\begin{eqnarray*}
\kappa_6&=& \frac{1}{\mu}\sum_{i=1}^{m-1}\alpha_i\int^{\eta_{i+1}}_{\eta_i} \bigg[-\frac{1}{5!}\int_{0}^{\tau}(\tau -s)^5 h(s)ds +\frac{\tau^5}{5!}\int_{0}^{1}(1-s)^3h(s)ds\bigg]d\tau\\
&&+
\frac{1}{\mu}\sum_{i=1}^{m}\beta_i\bigg[-\frac{1}{5!}\int_{0}^{\eta_i}(\eta_i-s)^5 h(s)ds +\frac{\eta_i^5}{5!}\int_{0}^{1}(1-s)^3h(s)ds\bigg]\\
&=& \frac{1}{\mu}\sum_{i=1}^{m-1}\alpha_i\int^{\eta_{i+1}}_{\eta_i} \bigg[\frac{1}{5!}\int_{0}^{\tau}[\tau^5(1-s)^3-(\tau-s)^5]h(s)ds+\frac{1}{5!}\int_{\tau}^{1}\tau^5(1-s)^3h(s)ds\bigg]d\tau\\
&&+
\frac{1}{\mu}\sum_{i=1}^{m}\beta_i \bigg[\frac{1}{5!}\int_{0}^{\eta_i}[\eta_i^5(1-s)^3-(\eta_i-s)^5]h(s)ds+\frac{1}{5!}\int_{\eta_i}^{1}\eta_i^5(1-s)^3h(s)ds\bigg].
\end{eqnarray*}
By inserting the value of  $\kappa_1$ and $\kappa_6$ in \eqref{eq005}, we obtain
\begin{eqnarray*}
u(t) &=& -\frac{1}{5!}\int_{0}^{t}(t-s)^5h(s)ds+ \frac{t^5}{5!}\int_{0}^{1}(1-s)^3h(s)ds\\
&&+
\frac{1}{\mu}\sum_{i=1}^{m-1}\alpha_i\int^{\eta_{i+1}}_{\eta_i} \bigg[\frac{1}{5!}\int_{0}^{\tau}[\tau^5(1-s)^3-(\tau-s)^5]h(s)ds+\frac{1}{5!}\int_{\tau}^{1}\tau^5(1-s)^3h(s)ds\bigg]d\tau\\
&&+
\frac{1}{\mu}\sum_{i=1}^{m}\beta_i \bigg[\frac{1}{5!}\int_{0}^{\eta_i}[\eta_i^5(1-s)^3-(\eta_i-s)^5]h(s)ds+\frac{1}{5!}\int_{\eta_i}^{1}\eta_i^5(1-s)^3h(s)ds\bigg]
\end{eqnarray*}
\begin{eqnarray*}
&=&\int_{0}^{t}\frac{1}{5!}[t^5(1-s)^3-(t-s)^5]h(s)ds+\int_{t}^{1}\frac{1}{5!}t^5(1-s)^3h(s)ds\\
&&+
\frac{1}{\mu}\sum_{i=1}^{m-1}\alpha_i\int^{\eta_{i+1}}_{\eta_i} \bigg[\frac{1}{5!}\int_{0}^{\tau}[\tau^5(1-s)^3-(\tau-s)^5]h(s)ds+\frac{1}{5!}\int_{\tau}^{1}\tau^5(1-s)^3h(s)ds\bigg]d\tau\\
&&+
\frac{1}{\mu}\sum_{i=1}^{m}\beta_i \bigg[\frac{1}{5!}\int_{0}^{\eta_i}[\eta_i^5(1-s)^3-(\eta_i-s)^5]h(s)ds+\frac{1}{5!}\int_{\eta_i}^{1}\eta_i^5(1-s)^3h(s)ds\bigg]\\
&=&\int_0^1\bigg[G(t,s) + \frac{1}{\mu}\sum_{i=1}^{m-1}\alpha_i\int^{\eta_{i+1}}_{\eta_i}G(\tau,s)d\tau+\frac{1}{\mu}\sum_{i=1}^{m}\beta_iG(\eta_i,s)\bigg]h(s)ds\\
&=&\int_0^1\bigg[G(t,s) +\frac{1}{\mu}\sum_{i=1}^{m-1}\alpha_i\Big( K(\eta_{i+1},s) - K(\eta_i,s)\Big)+\frac{1}{\mu}\sum_{i=1}^{m}\beta_iG(\eta_i,s) \bigg]h(s)ds,
\end{eqnarray*}
where $G$ and $K$ are defined by \eqref{eq004} and \eqref{eq004(2)} respectively. Conversely, we shall prove that the function $u$ defined by \eqref{eqsol} is solution of problem \eqref{eq003}-\eqref{eq002}. We can easily verify that the function $u$ satisfies the equation \eqref{eq003}, and also the first boundary condition \eqref{eq002}.\\
Now, for the second boundary condition we have,
\begin{eqnarray*}
\sum_{i=1}^{m}\beta_iu(\eta_i)&=&\sum_{i=1}^{m}\beta_i\Bigg[\int_0^1\bigg[G(\eta_{i},s) + \frac{1}{\mu}\sum_{i=1}^{m-1}\alpha_i\Big( K(\eta_{i+1},s) - K(\eta_i,s)\Big)\\
&&+\frac{1}{\mu}\sum_{i=1}^{m}\beta_iG(\eta_i,s) \bigg]h(s)ds \Bigg]\\
&=& \sum_{i=1}^{m}\beta_i\int_{0}^{1}G(\eta_{i},s)h(s)ds
+\frac{\sum_{i=1}^{m}\beta_{i}}{\mu}\sum_{i=1}^{m-1}\alpha_i\int_{0}^{1}\Big( K(\eta_{i+1},s) - K(\eta_i,s)\Big)h(s)ds\\
&&+\frac{\sum_{i=1}^{m}\beta_{i}}{\mu}\sum_{i=1}^{m}\beta_i\int_{0}^{1}G(\eta_{i},s)h(s)ds.
\end{eqnarray*}

\begin{eqnarray*}
\sum_{i=1}^{m-1}\alpha_i\int^{\eta_{i+1}}_{\eta_i} u(t)dt&=&\sum_{i=1}^{m-1}\alpha_i\int^{\eta_{i+1}}_{\eta_i}\bigg[\int_0^1\bigg[G(t,s) + \frac{1}{\mu}\sum_{i=1}^{m-1}\alpha_i\Big( K(\eta_{i+1},s) - K(\eta_i,s)\Big)\\
&&+\frac{1}{\mu}\sum_{i=1}^{m}\beta_iG(\eta_i,s) \bigg]h(s)ds \bigg]dt\\
&=&\sum_{i=1}^{m-1}\alpha_i\int_{0}^{1}\bigg[\int^{\eta_{i+1}}_{\eta_i}\bigg[G(t,s) + \frac{1}{\mu}\sum_{i=1}^{m-1}\alpha_i\Big( K(\eta_{i+1},s) - K(\eta_i,s)\Big)\\
&&+\frac{1}{\mu}\sum_{i=1}^{m}\beta_iG(\eta_i,s) \bigg]dt \bigg]h(s)ds\\
&=&\sum_{i=1}^{m-1}\alpha_i\int_{0}^{1}\bigg[\Big( K(\eta_{i+1},s) - K(\eta_i,s)\Big)+
\frac{\eta_{i+1}-\eta_{i}}{\mu}\\
&&\times\sum_{i=1}^{m-1}\alpha_i\Big( K(\eta_{i+1},s) - K(\eta_i,s)\Big)
+\frac{\eta_{i+1}-\eta_{i}}{\mu}\sum_{i=1}^{m}\beta_iG(\eta_i,s) \bigg]h(s)ds\\
&=& \sum_{i=1}^{m-1}\alpha_i\int_{0}^{1}\Big( K(\eta_{i+1},s) - K(\eta_i,s)\Big)h(s)ds+
\frac{\sum_{i=1}^{m-1}\alpha_{i}(\eta_{i+1}-\eta_{i})}{\mu}\\
&&\times\sum_{i=1}^{m-1}\alpha_i\int_{0}^{1}\Big( K(\eta_{i+1},s) - K(\eta_i,s)\Big)h(s)ds+\frac{\sum_{i=1}^{m-1}\alpha_{i}(\eta_{i+1}-\eta_{i})}{\mu}\\
&&\times\sum_{i=1}^{m-1}\beta_i\int_{0}^{1}G(\eta_i,s)h(s)ds.
\end{eqnarray*}
\end{proof}

We state some properties of Green's functions $G(t,s)$ and $K(t,s)$ that will be useful in this paper.\\
\begin{lemma} \label{lem3}
The Green's functions $G(t,s)$ and $K(t,s)$ satisfies the following properties:
\begin{itemize}
\item[(C1)] $G(t,s) \geq 0$ and $K(t,s) \geq 0$ for all $t,s \in J,$
\item[(C2)] $t^5G(1,s) \leq G(t,s) \leq G(1,s)$, for all $t,s \in J,$
\item[(C3)] $t^6K(1,s) \leq K(t,s) \leq K(1,s)$, for all $t,s \in J.$
\end{itemize}
\end{lemma}
\begin{proof}
\rm{(C1)} The nonnegativeness of $G(t,s)$ and $K(t,s)$ is obvious true for $t \leq s$, we only need to prove it for the case $0\leq s\leq t\leq 1$. Suppose that $0\leq s\leq t\leq 1$, then
\begin{eqnarray*}
G(t,s)&=& \frac{1}{5!}\big(t^5(1-s)^3-(t-s)^5\big)\\
&=& \frac{1}{5!}\big(t^2(t-ts)^3-(t-s)^5\big)\\
&\geq& \frac{1}{5!}\big(t^2(t-s)^3-(t-s)^5\big)\\
&=& \frac{1}{5!}(t-s)^3(t^2-(t-s)^2)\\
&=& \frac{s}{5!}(t-s)^3(2t-s) \geq 0.
\end{eqnarray*}

Similarly, we find
$$K(t,s)\geq \frac{1}{6!}s(t-s)^3\Big(t^{2}+t(t-s)+(t-s)^{2}\Big)\geq0,\ \text{for all}\ t,s\in J.$$

\rm{(C2)} For $0\leq s\leq t\leq 1$. If $s=1$, then $t=1$ and we have
\begin{eqnarray*}
G(t,s)=0=G(1,1)=t^5G(1,s).
\end{eqnarray*}
If $s=0$, we have
\begin{eqnarray*}
G(t,s)=0=t^5G(1,s).
\end{eqnarray*}
If $s=t$,
\begin{eqnarray*}
	G(t,s)=\frac{1}{5!}t^5(1-t)^3\geq \frac{1}{5!}t^5(1-t)^3-\frac{1}{5!}t^5(1-t)^5=t^5G(1,s).
\end{eqnarray*}
For $0<s<t\leq 1$, we have
\begin{eqnarray*}
\dfrac{G(t,s)}{G(1,s)} &=& \dfrac{t^5(1-s)^3-(t-s)^5}{(1-s)^3-(1-s)^5}\\
&\geq& \dfrac{t^5(1-s)^3-(t-ts)^5}{(1-s)^3-(1-s)^5}\\
&=&\dfrac{t^5\big((1-s)^3-(1-s)^5\big)}{(1-s)^3-(1-s)^5}\\
&=& t^5.
\end{eqnarray*}
Then
\begin{eqnarray*}
G(t,s)\geq t^5G(1,s).
\end{eqnarray*}
Now, putting
\begin{equation*}
H(t,s) =\frac{\partial G(t,s)}{\partial t} =\frac{1}{4!}
\begin{cases}
t^4(1-s)^3-(t-s)^4,& 0\leq s\leq t\leq 1 ,\\
t^4(1-s)^3,& 0\leq t\leq s\leq 1.\\
\end{cases}
\end{equation*}
For $0\leq s\leq t \leq 1$, then
\begin{eqnarray*}
4! H(t,s)&=& t^4(1-s)^3-(t-s)^4\\
&=& t(t-ts)^3-(t-s)^4\\
&\geq& t(t-s)^3-(t-s)^4\\
&=& s(t-s)^3\geq 0.
\end{eqnarray*}
Clearly, $H(t,s)\geq 0$ for $0\leq t\leq s \leq 1$. Thus, $G(t,s)$ is an increasing function with respect to $t$. As a consequence, we get
\begin{eqnarray*}
G(t,s) \leq G(1,s),\ \text{for all}\ t,s\in J.
\end{eqnarray*}
For $t\leq s$, we have\\
\begin{eqnarray*}
G(t,s)= \frac{1}{5!}t^5(1-s)^3
&\leq& \frac{1}{5!}s(1-s)^3\\
&\leq& \frac{1}{5!}s(1-s)^3(2-s)\\
&=& G(1,s),
\end{eqnarray*}
and
\begin{equation*}
G(t,s) = \frac{1}{5!}t^5(1-s)^3 \geq \frac{1}{5!}t^5\big((1-s)^3-(1-s)^5\big)= t^5G(1,s).
\end{equation*}
Therefore
\begin{equation*}
t^5G(1,s) \leq G(t,s) \leq G(1,s),\ \forall t,s \in J.
\end{equation*}
\rm{(C3)} Now, by adapting the same arguments developed to show \rm{(C2)}, we can prove that
\begin{equation*}
t^6K(1,s) \leq K(t,s) \leq K(1,s),\ \forall t,s \in J.
\end{equation*}
\end{proof}
\section{Main results}
In this section, we present our main results. To this end, let us consider the operator $A:X\rightarrow X$ defined as follows
\begin{equation} \label{eqOP}
Au(t)= \int_0^1\bigg[G(t,s) + \frac{1}{\mu}\sum_{i=1}^{m-1}\alpha_i\bigg( \int^{\eta_{i+1}}_{\eta_i}G(\tau,s)d\tau\bigg)+\frac{1}{\mu}\sum_{i=1}^{m}\beta_iG(\eta_i,s) \bigg]f(s,u(s))ds,\  t\in J,
\end{equation}
where $X=C(J,\mathbb{R})$ is endowed with the norm $\|u\|=\max_{t\in J}|u(t)|$.\\
Clearly, $u$ is a fixed point of the operator $A$ if and only if $u$ is a solution of the boundary value problem $\eqref{eq001}$-$\eqref{eq002}$.

 \begin{lemma} \label{lem3.1}
The operator $A$ is completely continuous.
\end{lemma}
\begin{proof} By using Arzela-Ascoli theorem with Lebesgue dominated convergence theorem, it is easy to show the complete continuity of the operator $A$.
\end{proof}

In the next theorem, we give sufficient conditions for the existence of at least one nontrivial solution of the problem $\eqref{eq001}$-$\eqref{eq002}$ which is the fixed point of the operator $A$ given by $\eqref{eqOP}$.

\begin{theorem}\label{thm1}
Assume that the following hypotheses are satisfied
\begin{itemize}
\item[(H1)] $f(t,u(t)) = p(t)g(u(t))$, where $p\in X$ and $g:\mathbb{R}\rightarrow \mathbb{R}$ is continuous with $\lim_{u\rightarrow \infty} \frac{g(u)}{u} = \gamma,\ \gamma\in\mathbb{R}.$\\
Furthermore, there exists $t_0\in(0,1]$ such that $f(t_0,0)=p(t_0)g(0)\neq 0$;
\item[(H2)] $\lvert \gamma\rvert \leq\frac{1}{M},$ where
\begin{equation*}
M=p^{*}\bigg(1+\frac{1}{\lvert \mu\rvert}\sum_{i=1}^{m-1}\lvert\alpha_i\rvert(\eta_{i+1}-\eta_i)+\frac{1}{\lvert \mu\rvert}\sum_{i=1}^{m}\lvert\beta_i\rvert\bigg)\int_0^1G(1,s)ds,\ \ p^{*}= \max_{t\in J}\lvert p(t)\rvert.
\end{equation*}
\end{itemize}
Then the boundary value problem \eqref{eq001}-\eqref{eq002} has at least one nontrivial solution  $u\in X.$
\end{theorem}
\begin{proof}
By Lemma \ref{lem3.1}, the operator $A$ is completely continuous.\\

We consider now, the linear boundary value problem

\begin{equation} \label{eqLin}
u^{(6)}(t) + \gamma p(t)u(t) = 0,\ t\in (0,1),
\end{equation}
with the boundary conditions \eqref{eq002}. Let the operator $L:X\rightarrow X$ defined as follows

\begin{equation}  \label{eqOPL}
Lu(t)= \int_0^1\bigg[G(t,s) + \frac{1}{\mu}\sum_{i=1}^{m-1}\alpha_i\Big( K(\eta_{i+1},s)- K(\eta_i,s)\Big)+\frac{1}{\mu}\sum_{i=1}^{m}\beta_iG(\eta_i,s) \bigg]\gamma p(s)u(s)ds,\  t\in J.
\end{equation}
Obviously, the fixed point of the bounded linear operator $L$ is a solution of the boundary value problem \eqref{eqLin}-\eqref{eq002} and conversely.
First, we will show that $1$ is not an eigenvalue for $L$. If, $\gamma=0$, then the Boundary value problem \eqref{eqLin}-\eqref{eq002} has no nontrivial solution. So, let us assume $\gamma\neq0$ and suppose on the contrary that the problem \eqref{eqLin}-\eqref{eq002} has a nontrivial solution $u\in X$ and $\|u\|>0$, then, by \rm{(H2)}, we have

\begin{eqnarray*}
\lVert u\rVert &=& \max_{t\in J}\lvert Lu(t)\rvert\\
&=& \max_{t\in J}\bigg\lvert  \int_0^1\bigg[G(t,s) + \frac{1}{\mu}\sum_{i=1}^{m-1}\alpha_i\bigg( \int^{\eta_{i+1}}_{\eta_i}G(\tau,s)d\tau\bigg)+\frac{1}{\mu}\sum_{i=1}^{m}\beta_iG(\eta_i,s) \bigg]\gamma p(s)u(s)ds\bigg\rvert \\
&\leq& \lvert \gamma\rvert p^{*}\lVert u\rVert \bigg(1+\frac{1}{\lvert \mu\rvert}\sum_{i=1}^{m-1}\lvert\alpha_i\rvert(\eta_{i+1}-\eta_i)+\frac{1}{\lvert \mu\rvert}\sum_{i=1}^{m}\lvert\beta_i\rvert\bigg)\int_0^1G(1,s)ds\\
&\leq& \lvert \gamma\rvert M\lVert u\rVert < \lVert u\rVert,
\end{eqnarray*}
which is a contradiction. Hence, $1$ is not an eigenvalue of $L$.

Next, we will show that \eqref{eqZK} is satisfied. Using hypothesis \rm(H1), for all $\varepsilon >0$ there exists $\rho>0$ such that
\begin{equation*}
\lvert g(u)-\gamma u\rvert<\varepsilon \lvert u\rvert,\ \text{for all}\ \lvert u\rvert> \rho.
\end{equation*}
Let $g^{\star}=\max_{|u|\leq\rho}|g(u)|$, so we can choose $\theta>0$ such that $\big(g^{\star}+|\gamma|\rho\big)<\varepsilon \theta$.\\
Now, let us define
\[J_{1}=\big\{t\in J:|u(t)|\leq \rho \big\}, \ J_{2}=\big\{t\in J:|u(t)|> \rho \big\}.\]
Then for all $u\in X$ with $\|u\|>\theta$, there hold the estimates
\[ \lvert g(u(t))-\gamma u(t)\rvert \leq |g(u(t))|+|\gamma||u(t)|\leq g^{\star}+|\gamma|\rho<\varepsilon \theta<\varepsilon \|u\|,\ \forall t\in J_{1}.\]
Similarly, for all $u\in X$ with $\|u\|>\theta$, we have
\[ \lvert g(u(t))-\gamma u(t)\rvert <\varepsilon \|u\|,\ \forall t\in J_{2}.\]
From \eqref{eqOP} and \eqref{eqOPL}, we get
\begin{eqnarray*}
\lvert Au(t)-Lu(t)\rvert &=& \bigg\lvert \int_0^1\bigg[G(t,s) + \frac{1}{\mu}\sum_{i=1}^{m-1}\alpha_i \bigg( \int^{\eta_{i+1}}_{\eta_i}G(\tau,s)d\tau\bigg)+\frac{1}{\mu}\sum_{i=1}^{m}\beta_iG(\eta_i,s) \bigg]\\
&&\times \Big(p(s)g(u(s))-\gamma p(s)u(s)\Big)ds\bigg\rvert\\
&\leq& M \int_0^1\lvert g(u(s))-\gamma u(s)\rvert ds\\
&\leq& M\varepsilon\lVert u\rVert.
\end{eqnarray*}
Thus, we obtain
\[\lim_{\lVert u\rVert \rightarrow \infty}\frac{\lVert Au-Lu\rVert}{\lVert u\rVert} =0.\]
Now, if we take $u=0$, then $0^{(6)}=-p(t)g(0)=0,\ t\in J$, which contradicts the hypothesis \rm{(H1)}.
Consequently, by Krasnoselskii-Zabreiko's fixed point Theorem \ref{Thm Kras-Zabr}, $A$ has at least one fixed point $u\in X$ which is a nontrivial solution of the boundary value problem \eqref{eq001}-\eqref{eq002}.
\end{proof}

Next, we present our second existence result of unique solution to the problem \eqref{eq001}-\eqref{eq002} which relies on Rus's fixed point theorem. \\

Now, for any $u_{1},\ u_{2}\in X=C(J,\mathbb{R})$, consider the following two metrics on $X$:
\begin{eqnarray*}
d(u_1,u_2) &=& \max_{t\in J}\lvert u_1(t) - u_2(t)\rvert,\\
\sigma(u_1,u_2) &=& \bigg(\int_0^1\lvert u_1(t) - u_2(t)\rvert^p dt\bigg)^{\frac{1}{p}},\ p>1.
\end{eqnarray*}
\begin{theorem}\label{thm2}
Let $f:J \times\mathbb{R}\rightarrow \mathbb{R}$ be a continuous function and assume that the following assumptions are satisfied
\begin{itemize}
\item[(H3)] There exists a constant $L>0$ such that $$\lvert f(t,u_1) - f(t,u_2)\rvert<L\lvert  u_1 - u_2\rvert,\ \forall t\in J,\ u_1, u_2 \in{\mathbb{R}},$$
\item[(H4)] $L\Phi<1,$
\end{itemize}
where
\begin{eqnarray*}
\Phi&=& \bigg(\int_0^1\lvert G(1,s)\rvert^pds\bigg)^{\frac{1}{p}} + \frac{1}{\lvert \mu\rvert}\sum_{i=1}^{m-1}\lvert\alpha_i\rvert\bigg(\int_0^1\Big\lvert\int^{\eta_{i+1}}_{\eta_i} G(\tau,s)d\tau\Big\rvert^p ds\Big)^{\frac{1}{p}}\\
&&+
\frac{1}{\lvert \mu\rvert}\sum_{i=1}^{m}\lvert\beta_i\rvert\bigg(\int_0^1\lvert G(\eta_i,s)\rvert^pds\bigg)^{\frac{1}{p}},\  p>1.
\end{eqnarray*}
Then the boundary value problem \eqref{eq001}-\eqref{eq002} has a unique solution.
\end{theorem}
\begin{proof} Here, one should use Theorem \ref{Thm Rus}.
By using H\"{o}lder and Minkowski inequalities, for all $ u_1, u_2 \in{\mathbb{R}}$ and $t\in J$, we have
\begin{eqnarray*}
\lvert (Au_1)(t) - (Au_2)(t)\rvert &=& \bigg\lvert \int_0^1\bigg[G(t,s) + \frac{1}{\mu}\sum_{i=1}^{m-1}\alpha_i\bigg( \int^{\eta_{i+1}}_{\eta_i} G(\tau,s)d\tau\bigg)+\frac{1}{\mu}\sum_{i=1}^{m}\beta_iG(\eta_i,s) \bigg]\\
&&\times
\big(f(s,u_1(s)) - f(s,u_2(s))\big) ds\bigg\rvert\\
&\leq&
\int_0^1\bigg\lvert G(t,s) + \frac{1}{\mu}\sum_{i=1}^{m-1}\alpha_i\bigg(\int^{\eta_{i+1}}_{\eta_i} G(\tau,s)d\tau\bigg)+\frac{1}{\mu}\sum_{i=1}^{m}\beta_iG(\eta_i,s)\bigg\rvert \\
&&\times
\lvert f(s,u_1(s)) - f(s,u_2(s))\rvert ds\\
&\leq&
L\int_0^1\bigg\lvert G(t,s) + \frac{1}{\mu}\sum_{i=1}^{m-1}\alpha_i\bigg( \int^{\eta_{i+1}}_{\eta_i} G(\tau,s)d\tau\bigg)+\frac{1}{\mu}\sum_{i=1}^{m}\beta_iG(\eta_i,s)\bigg\rvert\lvert u_1(s) - u_2(s)\rvert ds\\
&\leq&
L \bigg(\int_0^1\bigg\lvert G(t,s) + \frac{1}{\mu}\sum_{i=1}^{m-1}\alpha_i\Big( \int^{\eta_{i+1}}_{\eta_i} G(\tau,s)d\tau\Big)+\frac{1}{\mu}\sum_{i=1}^{m}\beta_iG(\eta_i,s)\bigg\rvert^p ds\bigg)^{\frac{1}{p}}\\
&&\times
\bigg(\int_0^1\lvert u_1(s) - u_2(s)\rvert^q ds\bigg)^{\frac{1}{q}}\\
&\leq&
L\bigg[\bigg(\int_0^1\lvert G(1,s)\rvert^pds\bigg)^{\frac{1}{p}} + \frac{1}{\lvert \mu\rvert}\sum_{i=1}^{m-1}\lvert\alpha_i\rvert\bigg(\int_0^1\Big\lvert\int^{\eta_{i+1}}_{\eta_i} G(\tau,s)d\tau \Big\rvert^pds\bigg)^{\frac{1}{p}}\\
&&+
\frac{1}{\lvert \mu\rvert}\sum_{i=1}^{m}\lvert\beta_i\rvert\bigg(\int_0^1\lvert G(\eta_i,s)\rvert^pds\bigg)^{\frac{1}{p}}\bigg]\sigma(u_1,u_2),
\end{eqnarray*}
then
\begin{equation*}
\lvert (Au_1)(t) - (Au_2)(t)\rvert \leq L\Phi\sigma(u_1,u_2).
\end{equation*}
Thus
\begin{equation*}
d(Au_1, Au_2) \leq \xi\sigma(u_1,u_2),
\end{equation*}
where $\xi = L\Phi>0.$

On the other hand, for all $u_{1},u_{2}\in X$, we have
\[d(Au_1, Au_2) \leq \xi\sigma(u_1,u_2)\leq \xi d(u_1, u_2) .\]
Then, for any $\varepsilon$ we can choose $\delta=\varepsilon \xi^{-1}$ so that $d(Au_1, Au_2)<\varepsilon$, whenever $d(u_1, u_2)<\delta$, which means that $A$ is continuous on $X$ with respect to the metric $d$.
Further, for all $u_{1},u_{2}\in X$, we get
\begin{equation*}
\int_0^1\lvert (Au_1)(t) - (Au_2)(t)\rvert^qdt \leq (L\Phi \sigma(u_1,u_2))^q,
\end{equation*}
and hence
\begin{equation*}
\sigma(Au_1, Au_2) \leq \nu\sigma(u_1,u_2),
\end{equation*}
where $\nu = L\Phi, 0<\nu<1.$\\
Thus, $A$ is contractive on $X$ with respect to the metric $\sigma$. Consequently, by Theorem \ref{Thm Rus}, $A$ has a unique fixed point in $X$ which is the unique (nontrivial) solution of the boundary value problem \eqref{eq001}-\eqref{eq002}.
\end{proof}

\section{Examples} In this section, we give two examples illustrating the usefulness of our theoretical
results.

\begin{exmp} \label{exp1}
We consider the following boundary value problem

\begin{equation}\label{exmple1/3eq001}
 \begin{cases} u^{(6)}(t) + t+10^{3}\arctan u= 0,\ t\in(0,1),\\
u'(0)= u''(0)= u^{(3)}(0)= u^{(4)}(0)= u''(1)=0,\\
u(0)= \int^{\frac{1}{3}}_{\frac{1}{4}} u(s)ds + 3u(\frac{1}{4})+4u(\frac{1}{3}),
\end{cases}
\end{equation}
where
$ f(t,u) = t+10^{3}\arctan u \in C(J\times\mathbb{R}, \mathbb{R})$ and $\alpha = 1, \beta_1=3, \beta_2= 4, \eta_1= \frac{1}{4}, \eta_2= \frac{1}{3}$ such that $\mu = -\frac{73}{12}.$\\
For all $u_1, u_2\in \mathbb{R}$, and $t\in J$, we have
\begin{eqnarray*}
\lvert f(t,u_1) - f(t,u_2)\rvert &=& \lvert 10^{3}(\arctan u_1 - \arctan u_2)\rvert\\
&&\leq L\lvert u_1 - u_2\rvert,
\end{eqnarray*}
 where $L = 10^{3}$.

On the other hand, by choosing $p=2$, we obtain
\begin{eqnarray*}
\Phi&=& \bigg(\int_0^1\lvert G(1,s)\rvert^2ds\bigg)^{\frac{1}{2}} + \bigg|\frac{\alpha}{\mu}\bigg| \bigg(\int_0^1\lvert K(\eta_{2},s) - K(\eta_1,s)\rvert^2ds\bigg)^{\frac{1}{2}}\\
&&+
\frac{1}{\lvert \mu\rvert}\sum_{i=1}^{2}\lvert\beta_i\rvert\bigg(\int_0^1\lvert G(\eta_i,s)\rvert^2ds\bigg)^{\frac{1}{2}}\\
&=& \bigg(\int_0^1\lvert G(1,s)\rvert^2ds\bigg)^{\frac{1}{2}} + \bigg|\frac{\alpha}{\mu}\bigg| \bigg(\int_0^{\eta_1}\lvert K(\eta_{2},s) - K(\eta_1,s)\rvert^2ds+ \int_{\eta_1}^{\eta_2}\lvert K(\eta_{2},s) - K(\eta_1,s)\rvert^2ds\\
&&+
\int_{\eta_2}^{1}\lvert K(\eta_{2},s) - K(\eta_1,s)\rvert^2ds\Big)^{\frac{1}{2}}+\frac{1}{\lvert \mu\rvert}\bigg[\lvert\beta_1\rvert\bigg(\int_0^1\lvert G(\eta_1,s)\rvert^2ds\bigg)^{\frac{1}{2}} + \lvert\beta_2\rvert\bigg(\int_0^1\lvert G(\eta_2,s)\rvert^2ds\bigg)^{\frac{1}{2}}\bigg]\\
&=& \bigg(\int_0^1\lvert G(1,s)\rvert^2ds\bigg)^{\frac{1}{2}} + \frac{12}{73}\bigg(\int_0^{\frac{1}{4}}\Big\lvert K(\frac{1}{3},s) - K(\frac{1}{4},s)\Big\rvert^2ds+ \int_{\frac{1}{4}}^{\frac{1}{3}}\Big\lvert K(\frac{1}{3},s) - K(\frac{1}{4},s)\Big\rvert^2ds\\
&&+
\int_{\frac{1}{3}}^{1}\Big\lvert K(\frac{1}{3},s) - K(\frac{1}{4},s)\Big\rvert^2ds\bigg)^{\frac{1}{2}}+
\frac{12}{73}\bigg[3\bigg(\int_0^1\Big\lvert G(\frac{1}{4},s)\Big\rvert^2ds\bigg)^{\frac{1}{2}} +4\bigg(\int_0^1\Big\lvert G(\frac{1}{3},s)\Big\rvert^2ds\bigg)^{\frac{1}{2}}\bigg]\\
&\approx& 0,000902884.
\end{eqnarray*}
Therefore, $L\Phi <1$.\\
Hence, based on Theorem \ref{thm2}, the boundary value problem \eqref{exmple1/3eq001} has a unique solution.
\end{exmp}
\begin{exmp} \label{exp2}
Consider the following boundary value problem

\begin{equation}\label{exmple2/3eq002}
\begin{cases} u^{(6)}(t) + \frac{10t(1+150u^3+\sin u)e^{-t^2}}{1+2u^2}= 0,\ t\in(0,1),\\
u'(0)= u''(0)= u^{(3)}(0)= u^{(4)}(0)= u''(1)=0,\\
u(0)= \int^{\frac{2}{3}}_{\frac{1}{2}} u(s)ds +2\int^{\frac{4}{5}}_{\frac{2}{3}} u(s)ds +\frac{1}{3}u(\frac{1}{2})+\frac{2}{5}u(\frac{2}{3})+ \frac{1}{4}u(\frac{4}{5}),
\end{cases}
\end{equation}
where $f(t,u(t)) = p(t)g(u(t))$ with $p(t)=10te^{-t^2} \in X$ and $g(u) = \frac{1+150u^3+\sin u}{1+2u^2}\in C(\mathbb{R},\mathbb{R})$. $\alpha_1=1, \alpha_2=2, \beta_1=\frac{1}{3}, \beta_2=\frac{2}{5},\beta_3=\frac{1}{4},\eta_1=\frac{1}{2},\eta_2=\frac{2}{3}$ and $\eta_3=\frac{4}{5}$.

We have $\mu=-\frac{5}{12}$, $\gamma=\lim_{u\rightarrow \infty} \frac{g(u)}{u}=75,\  p^{*}= \max_{t\in J}\lvert p(t)\rvert=5\sqrt{\frac{2}{e}}$ and $p(t_0)g(0)=10t_0e^{-t_0^2}\neq 0,$\ for some $t_0\in(0,1].$\\
By simple calculation, we find
\begin{eqnarray*}
M &=& p^{*}\Big(1+\frac{1}{\lvert \mu\rvert}\sum_{i=1}^{m-1}\lvert\alpha_i\rvert(\eta_{i+1}-\eta_i)+\frac{1}{\lvert \mu\rvert}\sum_{i=1}^{m}\lvert\beta_i\rvert\Big)\int_0^1G(1,s)ds\\
&=& 22\sqrt{\frac{2}{e}}\int_0^1G(1,s)ds\\
&=& \frac{11}{720}\sqrt{\frac{2}{e}}.
\end{eqnarray*}
Hence $\lvert \gamma\rvert < M^{-1}$. Since all hypotheses of Theorem \ref{thm1} are verified, therefore the boundary value problem \eqref{exmple2/3eq002} has at least one nontrivial solution.
\end{exmp}


\begin{thebibliography}{99}
\bibitem{Akram} Akram, G., Ur Rehman, H.: Solutions of a class of sixth order boundary value problems using the reproducing kernel space.  Abstr. Appl. Anal. 1-8 (2013)

\bibitem{Book1} Aris, O.: Mathematical theory of diffusion and reaction in permeable catalysts, Clarendon, Oxford (1975)

\bibitem{Ballem} Ballem, S., Viswanadham, K.: Numerical solution of sixth order boundary value problems by Galerkin method with quartic B-splines: Select Proceedings of NHTFF 2018. In book: Numerical Heat Transfer and Fluid Flow (pp.505-510)

\bibitem{bek1} Bekri. Z., Benaicha, S.: Nontrivial solution of a nonlinear sixth-order boundary value problem. Waves, Wavelets and fractals. 4(1), 10-18 (2018)

\bibitem{bek2} Bekri, Z. Benaicha, S.: Positive solutions for boundary value problem of sixth-order elastic beam equation. Open Journal of Mathematical Science, 4(1), 9-17 (2020)

\bibitem{Ben} Benaicha, S., Haddouchi, F.: Positive solutions of a nonlinear fourth-order integral boundary value problem. An. Univ. vest Timis. Ser. Mat. Inform. 54(1), 73-86 (2016)

\bibitem{Bon1} Bonanno, G., Candito, P., O’Regan, D.: Existence of nontrivial solutions for sixth-order differential equations. Math. Comput. Appl. Mdpi. Mathematics. 9(16), 1852 (2021)

\bibitem{Bon2} Bonanno, G., Livrea, R.: A sequance of positive solutions for sixth-order ordinary nonlinear differential problems. Electron. J. Qual. Theory Differ. Equ. 20, 1-17 (2021)

\bibitem{Che} Che Hussin, C, H., Mandangan, A., Kilicman, A., Daud, M, A., Juhan, N.: Differential transformation mathod for solving sixth-order boundary value problems of ordinary differential equations. Jurnal Teknologi, 78, 6-4 (2016)

\bibitem{Book5} Dulacska, E.: Soil Settlement Effects on Building, Developments in geotechnical engineering. Elsevier, Amsterdam, 69 (1992)

\bibitem{Guen} Guendouz, C., Haddouchi, F., Benaicha, S.: Existence of positive solutions for a nonlinear third-order integral boundary value problem. Ann. Acad. Rom. Sci. Ser. Math. Appl. 10(2), 314-328 (2018)

\bibitem{Hadd4} Haddouchi, F.: A note on existence results for a nonlinear fourth-order integral boundary value problem. Bul. Acad. Stiinte Repub. Mold. Mat. 91(3), 3–9 (2019)

\bibitem{Hadd5}  Haddouchi, F., Benaicha, S.: Some remarks on positive solutions of nonlinear problems at resonance. Journal of Informatics and Mathematical Sciences. 3(3), 291-293 (2011)

\bibitem{Hadd1} Haddouchi, F., Benaicha, S.: Positive solutions of a nonlinear three-point eigenvalue problem with integral boundary conditions. Rom. J. Math. Comput. Sci. 5(2), 202-2013 (2015)

\bibitem{Hadd3} Haddouchi, F., Guendouz, C., Benaicha, S.: Existence and multiplicity of positive solutions to a fourth-order multi-point boundary value problem. Mat. Vesnik. 73 (1), 25–36 (2021)

\bibitem{Hadd2} Haddouchi, F., Houari, N.: Monotone positve solutions of fourth order boundary value problem with mixed integral and multi-point boundary conditions. J. Appl. Math. Comput. 66 1(2), 87-109 (2021)

\bibitem{Kras}  Krasnosel'skii, M. A., Zabreiko, P. P.: Geometrical methods of nonlinear analysis, New York,
Springer, (1984)


\bibitem{Li2} Li, W.: The existence and multiplicity of positive solutions of nonlinear sixth-order boundary value problems with three variables coefficients. Bound. Value Probl. 22 (2012)

\bibitem{Li1} Li, W., Zhang, L., An, Y.: The Existence of positive solutions for a nonlinear sixth-order boundary value problem. ISRN Appl. Math. 2012, Art. ID 926952, 12 pp

\bibitem{Liu} Lui, Y., Weigho, Z., Chunfang, S.: Monotone and convex positive solutions for fourth-order multi-point boundary value problems. Bound. Value Probl. 21 (2011)

\bibitem{Mart} Martinez, A. L. M., Pendeza Martinez, C. A., Bressan, G. M., Souza, R. M., Stiegelmeir, E. W.: Multiple solutions for a sixth order boundary value problem. Trends Comput. Appl. Math. 22(1), 1-12 (2021)

\bibitem{Book3} Muray, J. D.: Mathematical biology. Biomathematics Texts, Springer-Verlag, 19 (1989)

\bibitem{Odda1} Odda, S. N.: Existence of solution for 5th order differential equations under some conditions. Appl. Math. 1, 279-282 (2010)

\bibitem{Odda2} Odda, S. N.: Positive solutions for nth order differential equations under some conditions. Appl. Appl. Math. 6(1), 1973-1980 (2011)

\bibitem{Panos} Palamides, P. K., Papageorgiou, E. H.: Approach to a fifth-order boundary value problem, via Sperner's Lemma. Appl. Math. 2(8), 993-998 (2011)

\bibitem{Book4} Prescott, J.: Applied Elasticity. Dover, New York (1993)

\bibitem{Rus} Rus, I. A.: On a fixed point theorem of maia type, Stud. Univ. Babes-Bolyai Math., 22
, 40-42 (1977)

\bibitem{Sidiq1} Siddiqi, S. S., Akram, G.: Solution of fifth order boundary value problems using nonpolynomial spline technique. Appl. Math. Comput. 175(2), 1574-1581 (2006)

\bibitem{Sidiq2} Siddiqi, S. S., Akram, G.: Septic spline solutions of sixth-order boundary value problems. Comput. Appl. Math. 215, 288-301 (2008) 	

\bibitem{ter} Tersian, S., Chaparova, J.: Periodic and Homoclinic solutions of some semilinear sixth-order differential equations. J. Math. Anal. Appl. 272(1), 223-239 (2002)

\bibitem{Book2} Thomas, L. H.: The calculation of atomic fields. Proc. Camb. Phil. Soc. 23, 542-548 (1927)

\bibitem{yang} Yang, B.: Positive solutions to a nonlinear sixth order boundary value problem. Differ. Equ. Appl. 11(2), 307-317 (2019)
\end{thebibliography}
\end{document}